\documentclass[]{amsart}

\usepackage{verbatim} 
\usepackage{amssymb}
\usepackage[foot]{amsaddr}
\usepackage{amsrefs}
\usepackage{amsthm}
\usepackage[multiple]{footmisc}

\newtheorem{thm}{Theorem}[section]
\newtheorem{cor}[thm]{Corollary}
\newtheorem{lem}[thm]{Lemma}
\newtheorem{prop}[thm]{Proposition}

\theoremstyle{definition}
\newtheorem{defin}{Definition}[section]

\newtheorem{rem}[thm]{Remark}
\newtheorem{question}[thm]{Question}

\numberwithin{equation}{thm}

\begin{document}
	
\title{A Recurrence For An Expression Involving Double Factorials}
\author{Joseph E. Cooper III}
\email{jecooper@alum.mit.edu}
\date{}

\begin{abstract}
We find a closed-form solution to the recurrence $\ z_{n+2} = \frac{1}{z_{n+1}} + z_n$, where $n \in \mathbb Z_{\geq 1}$ and $z_1 \in \mathbb R_{> 0},\  z_2 \in \mathbb R_{> 0}$. As a corollary, we derive an alternate proof of a recurrence for an expression involving double factorials which first appeared in the 2004 Putnam Contest. We subsequently pose several open questions suitable for motivated undergraduate students.
\end{abstract}
\maketitle

\section{Recurrence Relation}

The general recurrence we wish to solve is:

\begin{equation}\label{GeneralRecurrence}
z_1 \in \mathbb R_{> 0},\  z_2 \in \mathbb R_{> 0},\ z_{n+2} = \frac{1}{z_{n+1}} + z_n.
\end{equation}

The solution will be shown to be (for $n \geq 3$ and $\epsilon(n)$ defined in ~\ref{epsmod})

\begin{equation}
z_n = z_{\epsilon(n)}\cdot\displaystyle\prod\limits_{k=1}^{\lceil \frac{n}{2} \rceil - 1}\frac{n - 2k + z_1 z_2}{n - 2k - 1 + z_1 z_2}
\end{equation}

As a corollary we obtain another proof of
\begin{equation}\label{DoubleFactorialPutnam}
\forall n \in \mathbb Z_{\geq 3}, \frac{n!!}{(n-1)!!} = \frac{1}{\left(\frac{(n-1)!!}{(n-2)!!}\right)} + \frac{(n-2)!!}{(n-3)!!}, 
\end{equation} which is a result originally coming from Problem A3 of the 2004 Putnam Contest, according to \cite{oeis3}.

Not only does equation ~\ref{GeneralRecurrence} bear a vague similarity to the one defining the Fibonacci numbers, but the sequence we derive will conveniently satisfy an identity similar to Catalan's Identity for the Fibonacci numbers.

The ratio $\frac{(n+1)!!}{(n)!!}$ arose in \cite{jansontraveling}, wherein it is shown that the ratio's reciprocal (up to a constant factor) is equal to the average distance from a random point in the unit n-sphere along a random ray to the intersection of that ray with the sphere.

To address the proposition, we'll first need to develop a few tools.

\begin{defin}\label{epsmod} For $n \in \mathbb Z$ let
	\begin{equation*}
		\epsilon(n) = 
		\begin{cases} 
			2 & n \equiv 0 \pmod{2}
			\\
			1 & n \equiv 1 \pmod{2}
		\end{cases}
	\end{equation*}.
\end{defin}

\begin{defin}\label{DefnGGamma} Let $\Gamma_1 = z_1$ (from eqn. \ref{GeneralRecurrence}), $\Gamma_2 = z_2$. 
	
	For $j \in \mathbb Z_{\geq 1}$ let
	\begin{equation*}
		G(j) = \frac{j + \Gamma_1\Gamma_2}{j - 1 + \Gamma_1\Gamma_2}
	\end{equation*}
	
	and for $n \in \mathbb Z_{\geq 3}$ let
	\begin{equation*}
		\large\Gamma_n = \Gamma_{\epsilon(n)}\cdot\displaystyle\prod\limits_{k=1}^{\lceil \frac{n}{2} \rceil - 1} G(n - 2k)		
	\end{equation*}.
\end{defin}

\begin{rem}\label{RemPositive}
	Note that $G(j) > 0$ and therefore by closure of $\mathbb R_{> 0}$ under multiplication it follows that $\Gamma_n > 0$.
\end{rem}

Then our proposed solution to (\ref{GeneralRecurrence}) can be stated as

\begin{prop}[Translated Double Factorial Recurrence]\label{MainProp}
	$\{\Gamma_n\}_{n \geq 1}$ as defined in (\ref{DefnGGamma}) solves the recurrence (\ref{GeneralRecurrence}). That is, letting $\Gamma_1 = z_1 \in \mathbb R_{> 0}$ and letting $\Gamma_2 = z_2 \in \mathbb R_{> 0}$, then we have
	$\Gamma_{n} = \frac{1}{\Gamma_{n-1}} + \Gamma_{n-2}$ $(\forall n \in \mathbb Z_{\geq 3})$.
\end{prop}

\begin{lem}\label{EqnRecursiveDefnGamma}
First, notice that $(\forall n \in \mathbb Z_{\geq 3})$, $\Gamma_{n} = G(n-2)\cdot \Gamma_{n-2}$
\end{lem}
\begin{proof}
Since $ \lceil \frac{n}{2} \rceil - 2 = \lceil \frac{n - 2}{2} \rceil - 1 $ and $\epsilon(n - 2) = \epsilon(n)$,  
it follows that for $n \geq 3$, 

\begin{alignat}{2}
	\Gamma_{n} &= \Gamma_{\epsilon(n)}\cdot\prod\limits_{k=1}^{\lceil \frac{(n)}{2} \rceil - 1} G(n - 2k) \nonumber\\
	& = \Gamma_{\epsilon(n-2)}\cdot G(n-2)\cdot\prod\limits_{k=1}^{\lceil \frac{(n - 2)}{2} \rceil - 1} G((n-2) - 2k) \nonumber\\
	& = G(n-2)\cdot \left(\Gamma_{\epsilon(n-2)} \cdot\prod\limits_{k=1}^{\lceil \frac{(n - 2)}{2} \rceil - 1} G((n-2) - 2k)\right) \nonumber\\
	& = G(n-2)\cdot \Gamma_{n-2}  \nonumber
\end{alignat}
\end{proof}

The proof of the proposition then hinges on a Catalan-type identity satisfied by consecutive elements of the $\Gamma_*$ sequence.

\begin{lem}\label{ProductIdentity}
	$\forall n \geq 2$, $\Gamma_n\Gamma_{n-1} = n - 2 + \Gamma_1\Gamma_2$.
\end{lem}

\begin{proof}
	Proceed by induction; we prove two base cases: firstly, for n = 2, 	
	\[	\Gamma_2\Gamma_1 = 2 - 2 + \Gamma_2\Gamma_1\]	 	
	Secondly, for n = 3,	
	\[	\Gamma_3\Gamma_2 = \left(\frac{1+\Gamma_1\Gamma_2}{\Gamma_1\Gamma_2}\Gamma_1\right)\Gamma_2 = 1 + \Gamma_1\Gamma_2 = (3 - 2) + \Gamma_1\Gamma_2\]	
	Now assume that the statement is true for $n = N$, where $N \in \mathbb Z_{\geq 3}$. Then for $n = N+1$,	
	\begin{alignat}{2}\label{pfeq1}
		\Gamma_{N+1}\Gamma_{N} & = \left(G(N-1)\cdot \Gamma_{N-1}\right)\left(G(N-2)\cdot \Gamma_{N-2}\right)  \\ 
		& \text{(by Lemma $\ref{EqnRecursiveDefnGamma}$)} \nonumber \\
		& = \left(G(N-1)\cdot G(N-2)\right)\left( \Gamma_{N-1} \cdot \Gamma_{N-2}\right)  \\
		& = \left(\frac{N-1+\Gamma_1\Gamma_2}{N-2+\Gamma_1\Gamma_2} \cdot \frac{N-2+\Gamma_1\Gamma_2}{N-3+\Gamma_1\Gamma_2}  \right)\left( \Gamma_{N-1} \cdot \Gamma_{N-2}\right)  \\
		& = \left(\frac{N-1+\Gamma_1\Gamma_2}{N-3+\Gamma_1\Gamma_2}  \right)\left( \Gamma_{N-1} \cdot \Gamma_{N-2}\right)  \\		
		& = \left(\frac{N-1+\Gamma_1\Gamma_2}{N-3+\Gamma_1\Gamma_2}  \right)\left( N-3 + \Gamma_1\Gamma_2 \right)  \\
		& \text{(per our induction assumption)} \nonumber	\\	
		& = \left(N-1+\Gamma_1\Gamma_2\right)  
	\end{alignat}
\end{proof}

We now have all the tools we need in order to prove the proposition.

\begin{proof}[Proof of Prop. \ref{MainProp}]
	
	For $n \in \mathbb Z_{\geq 3}$,
	
	\begin{alignat}{2}\label{pfeq1}
		\frac{1}{\Gamma_{n-1}} + \Gamma_{n-2} & = \Gamma_{n-2} \cdot \left(\frac{1}{\Gamma_{n-1}\Gamma_{n-2}} + 1 \right) \\ 
		& \text{(Recall $\Gamma_* > 0$ (Remark $\ref{RemPositive}$))} \nonumber	\\		
		& = \Gamma_{n-2} \cdot \left(\frac{1}{n-3 + \Gamma_{1}\Gamma_{2}} + 1 \right) \\ 
		& \text{(by ($\ref{ProductIdentity}$))} \nonumber\\ 
		& = \Gamma_{n-2} \cdot \left(\frac{n-2 + \Gamma_{1}\Gamma_{2}}{n-3 + \Gamma_{1}\Gamma_{2}} \right) \\ 				
		& = \Gamma_{n-2} \cdot G(n-2)  \\
		& = \Gamma_{n}  \\		
		& \text{(by ($\ref{EqnRecursiveDefnGamma}$))} \nonumber		
	\end{alignat}
\end{proof}

\begin{question}
	Can the definition of $\Gamma_n$ be extended to $\Gamma_x$ (where $x$ may be any positive real number) such that $\Gamma_x$ is continuous, while still satisfying the recurrence formula (replacing $n$ with $x$)?
\end{question}

\begin{question}
	If such a function exists, does it have any other interesting properties or satisfy any other identities?
\end{question}

\begin{question}
	What if we additionally require the function to be differentiable on its domain?
\end{question}

\section{Double Factorial Expression}

It's now time to consider the expression $\frac{(n+1)!!}{n!!}$ (recall that $(2n)!! = (2n)\cdot(2n-2)\cdot\cdot\cdot2$ and $(2n+1)!! = (2n+1)\cdot(2n-1)\cdot\cdot\cdot1$, and that $0!! = 1$ ).

It just so happens that letting $\Gamma_1 = 1$ and $\Gamma_2 = 1$ yields the desired sequence of expressions. That is, 

\[ G(j) = \frac{j + \Gamma_1\Gamma_2}{j - 1 + \Gamma_1\Gamma_2} \]

becomes (using \textprime-notation to remind ourselves that we have made substitutions)

\[ G\textprime(j) := \frac{j + 1\cdot1}{j - 1 + 1\cdot1} = \frac{j + 1}{j} \],

and thus 

\[ \large\Gamma_n = \Gamma_{\epsilon(n)}\cdot\displaystyle\prod\limits_{k=1}^{\lceil \frac{n}{2} \rceil - 1} G(n - 2k) \]

becomes (noting that since we have substituted $\Gamma_1 = \Gamma_2 = 1$, it follows that $\Gamma_{\epsilon(n)} = 1$ $\forall n$)

\begin{alignat}{2}
	 \large\Gamma\textprime_n & = 1\cdot\displaystyle\prod\limits_{k=1}^{\lceil \frac{n}{2} \rceil - 1} G\textprime(n - 2k) \label{replace1}  \\ 		 
	 & = \displaystyle\prod\limits_{k=1}^{\lceil \frac{n}{2} \rceil - 1} \frac{(n-2k) + 1}{(n-2k)} \\	 
	 & = \frac{\prod\limits_{k=1}^{\lceil \frac{n}{2} \rceil - 1}(n-2k + 1)}{\prod\limits_{k=1}^{\lceil \frac{n}{2} \rceil - 1}(n-2k)} \\
	  & = \frac{(n-1)!!}{(n-2)!!} \label{replace2}
\end{alignat}.

Thus we have an alternate derivation of the double factorial formula:

\begin{cor}\label{mainprop}
	
	$\begin{aligned}
	\frac{n!!}{(n-1)!!} = \frac{1}{\left(\frac{(n-1)!!}{(n-2)!!}\right)} + \frac{(n-2)!!}{(n-3)!!}, & & \forall n \in \mathbb Z_{\geq 3}.
	\end{aligned}$
	
\end{cor}

\begin{proof}
	We may replace $\Gamma$ wherever it occurs in Proposition \ref{MainProp} with $\Gamma\textprime$ (this is simply the result of setting $\Gamma_1 = \Gamma_2 = 1$ as shown in \ref{replace1}\textendash\ref{replace2}).
\end{proof}

\end{document}